\documentclass[12pt, reqno]{amsart}
\usepackage{mathrsfs}
\usepackage{graphicx}
\usepackage{tikz}
\usepackage{amsmath}
\usepackage{amsfonts}
\usepackage{amssymb}
\usepackage{hyperref}

\usepackage{epic}

\newtheorem{definition}{Definition}
\newtheorem{theorem}{Theorem}
\newtheorem{lemma}[theorem]{Lemma}

\newtheorem{remark}{Remark}
\newtheorem{proposition}[theorem]{Proposition}
\newtheorem{corollary}[theorem]{Corollary}

\usepackage{amsthm,amsmath,amssymb,url,cite,color}

\numberwithin{equation}{section}

\def\inv{\mathrm{inv}}
\def\aid{\mathrm{aid}}
\def\ai{\mathrm{ai}}
\def\st{\mathrm{st}}

\def\Fix{\mathrm{FIX}}
\def\fix{\mathrm{fix}}

\def\rix{\mathrm{rix}}

\def\max{\mathrm{max}}

\def\exc{\mathrm{exc}}
\def\dd{\mathrm{dd}}
\def\des{\mathrm{des}}
\def\Des{\mathrm{DES}}
\def\maj{\mathrm{maj}}
\def\imaj{\mathrm{imaj}}

\def\S{\mathfrak{S}}
\def\R{\mathcal{R}}
\def\DE{\mathcal{E}}

\def\da{\mathrm{da}}

\def\SCF{\mathrm{SCF}}
\def\LG{{L-hook}}
\def\FG{{F-hook}}
\def\cyc{\mathrm{cyc}}
\def\lyc{\mathrm{lyc}}
\def\valley{\mathrm{valley}}
\def\peak{\mathrm{peak}}
\def\cda{\mathrm{cda}}

\def\Rix{\mathrm{RIX}}
\def\Orb{\mathrm{Orb}}

\def\imaj{\mathrm{imaj}}

\def\N{\mathbb N}
\def\Z{\mathbb Z}

\def\A{\mathfrak{A}}
\def\D{\mathfrak{D}}
\def\T{\mathcal{T}}

\begin{document}

\title[The $\gamma$-positivity  of 
basic Eulerian polynomials]{The $\gamma$-positivity  of
basic Eulerian polynomials\\ via group actions}

\author{Zhicong Lin}
\address[Zhicong Lin]{School of Sciences, Jimei University, Xiamen 361021,
P.R. China
}
\email{lin@math.univ-lyon1.fr}

\author{Jiang Zeng}
\address[Jiang Zeng]{Institut Camille Jordan,  Universit\'{e} Claude Bernard Lyon 1, France}
\email{zeng@math.univ-lyon1.fr}

\date{\today}
\begin{abstract}
We provide combinatorial interpretation for the $\gamma$-coefficients of 
 the basic Eulerian polynomials that enumerate permutations by 
 the excedance statistic and the major index as well as 
 the corresponding  $\gamma$-coefficients for derangements. Our results  refine the classical $\gamma$-positivity results for the Eulerian polynomials and the derangement polynomials.
 The main tools are  Br\"and\'en's  modified Foata--Strehl action on permutations and the recent  triple statistic ($\des, \rix,\aid$) equidistibuted with ($\exc, \fix, \maj$).
\end{abstract}
\maketitle
\section{Introduction}
Any polynomial $h(t)=\sum_{i=0}^nh_it^i$ with 
symmetric coefficient sequence ($h_{n-i}=h_i$) can be expanded uniquely as 
$h(t) =\sum_{k=0}^{\lfloor {n/2}\rfloor}\gamma_k t^k(1 + t)^{n-2k}$,
and  is said to be
{\em$\gamma$-positive} if  $\gamma_k\geq 0$.   
An interesting problem is then to see whether there are any combinatorial or geometric 
meaning of the $\gamma$-coefficients when they are positive, see \cite{br,fs, npt,prw}.
It is well konwn that 
Eulerian polynomials  are $\gamma$-positive, see \eqref{sym:unim} below. 
In this work we shall study two generalizations of the $\gamma$-positivity of Eulerian polynomials.

A permutation $\sigma$ of $[n]:=\{1,2,\ldots,n\}$
  is  a bijection from  $[n]$ to $[n]$, which will be identified with  the word $\sigma=\sigma_1\cdots\sigma_n$, where $\sigma_i=\sigma(i)$.  We  denote by $\S_n$ the set of all permutations of $[n]$.
 Given a  $\sigma\in\S_n$, an integer $i\in[n]$ is an \emph{excedance} (resp.~\emph{fixed point}) of $\sigma$ if $\sigma_i>i$ (resp.~$\sigma_i=i$). Denote by $\exc(\sigma)$ and $\fix(\sigma)$ the number of excedances and fixed points  of $\sigma$, respectively. Let $\maj(\sigma):=\sum_{\sigma_i>\sigma_{i+1}} i$ be the major index of $\sigma$.  
 Shareshian and Wachs~\cite{sw2,sw} introduced and 
 studied 
the  {\em basic Eulerian polynomials}
$$
A_{n}(t,r,q):=\sum_{\sigma\in \S_n}t^{\exc(\sigma)} r^{\fix(\sigma)}q^{\maj(\sigma)-\exc(\sigma)}.
$$
In \cite[Remark~5.5]{sw} they noticed  that using an 
unpublished result of Gessel they could prove, among other things,  that 
\begin{enumerate}
\item there are polynomials $\gamma_{n,k}(q)$ in  $\N[q]$ such that
\begin{equation}\label{gam:euler}
A_n(t, 1, q)=\sum_{k=0}^{{\lfloor n/2\rfloor}} \gamma_{n,k}(q) t^k(1+t)^{n-2k}.
\end{equation}
\item  there are polynomials $\widetilde{\gamma}_{n,k}(q)$ in  $\N[q]$ such that
\begin{equation}\label{gam:der}
A_n(t, 0, q)= \sum_{k=0}^{{\lfloor (n-1)/2\rfloor}} \widetilde{\gamma}_{n,k}(q) t^k(1+t)^{n-1-2k}.
\end{equation}
\end{enumerate}
They have also two  refined versions of \eqref{gam:der}, that we shall discuss  in the last section, see \eqref{exp:fixed} and \eqref{SW3}.

The aim of this paper is to provide combinatorial 
interpretations for the coefficients $\gamma_{n,k}(q)$ and $\widetilde{\gamma}_{n,k}(q)$. 
Several $q$-analogs of Eulerian polynomials with combinatorial meanings have been studied in the literature (see \cite{sw2} and references therein) and various extensions of~\eqref{sym:unim} and \eqref{eq:der} have already been obtained in \cite{br, sz, hjz, sz14}. 
We first review some partial and related 
recent results on this topic.

 Let $\sigma=\sigma_1\sigma_2\ldots \sigma_n$ be permutation 
  in $\S_n$ and by convention  $\sigma_0=\sigma_{n+1}=+\infty$.
   We call $\sigma_i$ ($1\leq i\leq n$) a 
   {\em double descent} (resp.~{\em double ascent}, 
   {\em peak}, {\em valley}) of $\sigma$ if
$
\sigma_{i-1}>\sigma_i>\sigma_{i+1}
$
(resp. $\sigma_{i-1}<\sigma_i<\sigma_{i+1}$, $\sigma_{i-1}<\sigma_i>\sigma_{i+1}$, $\sigma_{i-1}>\sigma_i<\sigma_{i+1}$). Denote by  $\dd(\sigma)$ (resp.~$\da(\sigma)$, $\peak(\sigma)$, $\valley(\sigma)$) the number of double descents (resp.~double ascents, peaks, valleys) of $\sigma$.  For each $\sigma\in\S_n$ let 
$\cda(\sigma):=|\{i : \sigma^{-1}(i)<i<\sigma(i)\}|$
be the number of {\em double excedances} or {\em cyclic double ascents} of $\sigma$ 
and $\cyc(\sigma)$ denote the number of cycles of $\sigma$.   For $1\leq k\leq n$ define the sets
\begin{align*}
\D_{n,k}:&=\{\sigma\in\S_n : \dd(\sigma)=0,\des(\sigma)=k\},\\
 \widetilde{\D}_{n,k}:&=\{\sigma\in \D_{n,k-1}: \sigma_{n-1}<\sigma_n\},\\
  \DE_{n,k}:&=\{\sigma\in\S_n: \fix(\sigma)=0,\cda(\sigma)=0,\exc(\sigma)=k\}.
\end{align*}

For example, for $n=4$, we have 
\begin{align}
\D_{4,0}&=\{1234\}\quad\text{and}\quad 
   \D_{4,1}=\{1324,1423,2314,2413,3412,1243,1342,2341\};\label{D}\\
  \widetilde{\D}_{4,1}&=\{1234\}\quad{and} \quad
  \widetilde{\D}_{4,2}=\{1324,1423,2314,2413,3412\};\label{DE}\\
  \DE_{4,1}&=\{4123\}\quad\text{and}\quad 
   \DE_{4,2}=\{2143,2413, 3412, 4312, 4321\}.\label{E}
   \end{align} 
A classical result of Foata and Sch\"utzenberger~\cite{fs} states that the Eulerian polynomials have the following $\gamma$-expansion, which implies  
both the {\em symmetry} and {\em unimodality} (see for instance~\cite{sw} for definitions) of the sequence of Eulerian numbers.
\begin{theorem}[Foata--Sch\"utzenberger]\label{euler} One has 
\begin{align}\label{sym:unim}
A_{n}(t,1,1)&=\sum_{k=0}^{\lfloor(n-1)/2\rfloor}|\D_{n,k}|t^k(1+t)^{n-1-2k}.
\end{align} 
 \end{theorem}
Shin and Zeng~\cite[Theorem~11]{sz}  have recently proved   the following result  via continued fractions.  
\begin{theorem}[Shin--Zeng]\label{theoremSZ} One has
\begin{equation}\label{cycle}
\sum_{\substack{ \sigma\in\S_n \\ \fix(\sigma)=0}}\beta^{\cyc(\sigma)}t^{\exc(\sigma)}=\sum_{k=1}^{\lfloor n/2\rfloor}\biggl(\:\sum_{\sigma\in \DE_{n,k}}\beta^{\cyc(\sigma)}\biggr)t^k(1+t)^{n-2k}.
\end{equation}
\end{theorem}
In particular, we have the derangement analogue of \eqref{sym:unim}
 \begin{align}
\label{eq:der}
A_n(t,0,1)&=\sum_{k=1}^{\lfloor n/2\rfloor}|\DE_{n,k}|t^k(1+t)^{n-2k}.
 \end{align}
For $n=4$, it follows from \eqref{E} that 
   \begin{align*}
\sum_{\substack{ \sigma\in\S_n \\ \fix(\sigma)=0}}\beta^{\cyc(\sigma)}t^{\exc(\sigma)}
   &=\beta t(1+t)^2+(2\beta+3\beta^2)t^2\\
   &=\beta t+4\beta t^2+3\beta^2t^2+\beta t^3.
   \end{align*}
   
A permutation $\sigma\in\S_n$ is called {\em alternating} if $\sigma_1<\sigma_2>\sigma_3<\sigma_4>\cdots$. 
An \emph{inversion} of a permutation $\sigma\in\S_n$ is a pair $(\sigma_i,\sigma_j)$ such that $1\leq i<j\leq n$ and $\sigma_i>\sigma_j$. Let $\inv(\sigma)$ be the number of inversions of $\sigma$. The set of alternating permutations of order $n$ is denoted by $\A_n$.  When $t=-1$
Foata and Han~\cite[Theorem~1]{fh} proved the following $q$-analogue of 
 \eqref{sym:unim} and \eqref{eq:der}.
\begin{theorem}[Foata-Han]\label{FH} One has  $A_{2n}(-1,1,q)=A_{2n-1}(-1,0,q)=0$ 
\textup{(}$n\geq1$\textup{)} and
\begin{align*}
A_{2n+1}(-1,1,q)&=(-1)^n
\sum_{\sigma\in\A_{2n+1}}q^{\inv(\sigma)}
\quad (n\geq0);\\
A_{2n}(-1,0,q)&=(-1)^n\sum_{\sigma\in\A_{2n}}q^{\inv(\sigma)}\quad (n\geq 0).
\end{align*}
\end{theorem}

Our main results, Theorems~\ref{main:result} and \ref{main:result2}, 
are $q$-analogues of \eqref{sym:unim} and \eqref{eq:der} for general $t$. 
  \begin{theorem}\label{main:result}
 We have
 \begin{align}
 A_n(t,1,q)=\sum_{k=0}^{\lfloor(n-1)/2\rfloor}\biggl(\:\sum_{\sigma\in \D_{n,k}}q^{\inv(\sigma)} \biggr)t^k(1+t)^{n-1-2k}.\label{per}
\end{align}
  \end{theorem}
For $n=4$,  it follows from \eqref{D} that 
   \begin{align*}
   A_4(t,1,q)
   &=(1+t)^3+(2q+3q^2+2q^3+q^4)t(1+t)\\
   &=1+(3+2q+3q^2+2q^3+q^4)t+(3+2q+3q^2+2q^3+q^4)t^2+t^3.
   \end{align*}
   \begin{theorem}\label{main:result2}
 We have
 \begin{align}
  A_n(t,0,q)=\sum_{k=1}^{\lfloor n/2\rfloor}\biggl(\:\sum_{\sigma\in \widetilde{\D}_{n,k}}q^{\inv(\sigma)}\biggr)t^k(1+t)^{n-2k}.\label{der2}
\end{align}
  \end{theorem}
  
For $n=4$, it follows from \eqref{DE} that 
   \begin{align*}
   A_4(t,0,q)&=t(1+t)^2+(q+2q^2+q^3+q^4)t^2\\
   &=t+(2+q+2q^2+q^3+q^4)t^2+t^3.
   \end{align*}
 
  Obviously Theorem~\ref{main:result} is a $q$-analogue of 
  Theorem~\ref{euler}.
 To show that  \eqref{der2} reduces to \eqref{eq:der}, we will give a bijection  
 from $\widetilde{\D}_{n,k}$ to $\DE_{n,k}$ in 
Remark~\ref{bijection}.
Since $\A_{2n+1}=\D_{2n+1,n}$ and $\A_{2n}= \widetilde{\D}_{2n,n}$,
 Theorems~\ref{main:result} and  \ref{main:result2} with $t=-1$ reduce immediately to
 Theorem~\ref{FH}.

   \begin{remark}    
   For any $\sigma\in \S_n$ let 
$\Des(\sigma)=\{i\in [n-1]: \sigma_i>\sigma_{i+1}\}$
be its descent set. 
According to a result of Foata and Sch\"utzenberger~\cite{fs2}, 
 for any $S\subseteq[n-1]$, 
 $$
 \sum q^{\inv(\sigma)}= \sum  q^{\imaj(\sigma)}\qquad \quad(\sigma\in\S_n \quad \textrm{and}\quad \Des(\sigma)=S),
 $$
 where $\imaj(\sigma)=\maj(\sigma^{-1})$.
 Hence we can replace the statistic $\inv$ by $\imaj$ in 
 Theorems~\ref{main:result} and~\ref{main:result2}.
 \end{remark}

 We will prove Theorems~\ref{main:result} and \ref{main:result2}   
 by making use of   Br\"and\'en's  modified Foata--Strehl  action, \cite{fsh, br} (see also\cite{swg}),  and an alternative interpretation of the basic Eulerian polynomials  involving the descent statistic, that we recall below.

\begin{definition}
  Let $w=w_1w_2\cdots w_k$ be a word of length $k$ with distinct letters from $[n]$.  An \emph{admissible inversion} of $w$ is a pair $(w_i,w_j)$ such that $1\leq i<j\leq k$ and $w_i>w_j$ and satisfies either of the following conditions:
 \begin{itemize}
 \item $1<i$ and $w_{i-1}<w_i$ or 
 \item there is some $l$ such that $i<l<j$ and $w_i<w_l$.
 \end{itemize}
Let $\ai(w)$ be the number of admissible inversions of $w$.
The statistic $``\rix"$ is defined 
  recursively as follows: $\rix(\emptyset)=0$, 
   if $w_i=\max\{w_1,w_2,\ldots,w_k\}$ \textup{(}$1\leq i\leq k$\textup{)}, then
$$
 \rix(w):=
 \begin{cases}
0,&\text{if $i=1< k$;}\\
1+\rix(w_1w_2\cdots w_{k-1}),&\text{if $i=k$};\\
\rix(w_{i+1}w_{i+2}\cdots w_k),& \text{if $1<i<k$.}
\end{cases}
$$
\end{definition}
  Note that $\rix(w_1)=1$. If $w=291753468$,  then there are $14$ inversions, except  $(5,3)$ and $(5,4)$,  all others  are admissible, hence  $\ai(w)=12$;
  for the computation of $\rix(w)$, by looking for
   the greatest letters  in the words successively
  we have 
   $$
   \rix(w)=\rix(175346{\bf8})=1+\rix(1{\bf7}5346)=1+\rix(534{\bf6})=2+\rix({\bf5}34)=2.
   $$
  
  We will need  the following
   interpretation of the basic Eulerian polynomials~\cite[Theorem~8]{lin} (see also \cite{bu} for an equivalent version), of which  the special  $r=1$ case was  first proved in 
\cite{lsw} using the Rees product of posets.
   \begin{lemma}[\hspace{-0.5pt}\cite{sw,lin,bu}] One has 
 \begin{equation}\label{des:ai}
 A_n(t,r,q)=\sum_{\sigma\in \S_n}t^{\des(\sigma)}r^{\rix(\sigma)}q^{\ai(\sigma)}.
 \end{equation}
 \end{lemma}

  The rest of this paper is organized as follows. 
  In Section~\ref{sec:main}, after recalling 
 the Modified Foata--Strehl  action, we prove
Theorem~\ref{main:result}.
   In Section~\ref{sec:fac}, we introduce  an alternative algorithm to 
   compute  the statistic $\rix$, which connects  another model $\R^{0}_{n,k}$ (see \eqref{newmodel}) to 
$\DE_{n,k}$.  Proof of Theorem~\ref{main:result2} will be given in Section~\ref{sec:main2}.
  We derive the recurrence relations and the generating functions of the $\gamma$-coefficients 
  in \eqref{per} and \eqref{der2} in Section~\ref{sec:gamma-coeff}.
   We conclude the paper with some further remarks in Section~\ref{sec:conclusion}.

\section{Proof of Theorem~\ref{main:result}}\label{sec:main}
Let $\sigma\in\S_n$, for  any  $x\in[n]$, the {\em$x$-factorization} of $\sigma$ reads
$
\sigma=w_1 w_2x w_3 w_4,
$
 where $w_2$ (resp.~$w_3$) is the maximal contiguous subword immediately to the left (resp.~right)  of $x$ whose letters are all smaller than $x$.  Following Foata and Strehl~\cite{fsh} we define the action $\varphi_x$ by
 $$
 \varphi_x(\sigma)=w_1 w_3x w_2 w_4.
 $$
 For instance, if $x=4$ and $\sigma=2743156\in\S_7$, then $w_1=27,w_2=\emptyset,w_3=31$ and $w_4=56$.
 Thus $\varphi_x(\sigma)=2731456$.
 Clearly, $\varphi_x$ is an involution acting on $\S_n$ and it is not hard to see that $\varphi_x$ and $\varphi_y$ commute for all $x,y\in[n]$. Br\"and\'en~\cite{br} modified $\varphi_x$ to be 
 $$ \varphi'_x(\sigma):=
 \begin{cases}
\varphi_x(\sigma),&\text{if $x$ is a double ascent or double descent of $\sigma$};\\
\sigma,& \text{if $x$ is a valley or a peak of $\sigma$.}
\end{cases}
$$
Again it is clear that $\varphi'_x$'s are involutions and  commute. For any subset $S\subseteq[n]$ we can then define the function $\varphi'_S :\S_n\rightarrow\S_n$ by 
$$
\varphi'_S(\sigma)=\prod_{x\in S}\varphi'_x(\sigma).
$$
Hence the group $\Z_2^n$ acts on $\S_n$ via the functions $\varphi'_S$, $S\subseteq[n]$. This action will be called  the {\em Modified Foata--Strehl action} ({\em MFS-action} for short)
 as depicted in Fig.~\ref{valhop}. We first show that the statistic $``\ai"$ is invariant under the MFS-action.

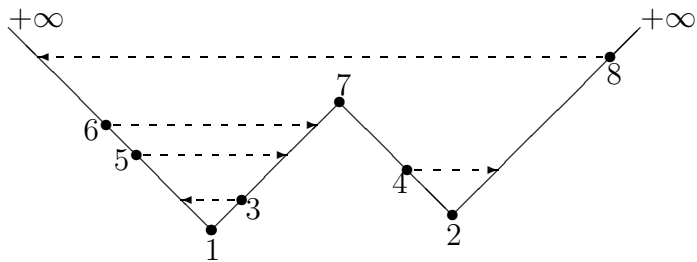
\begin{figure}
\setlength {\unitlength} {0.8mm}
\begin {picture} (90,40) \setlength {\unitlength} {1mm}
\thinlines
\put(24,8){\dashline{1}(-1,0)(-8,0)}
\put(16,8){\vector(-1,0){0.1}}

\put(10,14){\dashline{1}(1,0)(20,0)}
\put(30,14){\vector(1,0){0.1}}

\put(6,18){\dashline{1}(1,0)(28,0)}
\put(34,18){\vector(1,0){0.1}}

\put(46,12){\dashline{1}(1,0)(12,0)}
\put(58,12){\vector(1,0){0.1}}

\put(73,27){\dashline{1}(-1,0)(-76,0)}
\put(-3,27){\vector(-1,0){0.1}}

\put(20,4){\line(-1,1){27}}\put(-7,31){$+\infty$}
\put(10,14){\circle*{1.3}}\put(7,12){$5$}
\put(6,18){\circle*{1.3}}\put(3,16){$6$}
\put(20,4){\circle*{1.3}}
\put(20,4){\circle*{1.3}}\put(19.1,0){$1$}
\put(24,8){\circle*{1.3}}\put(24.5,5.5){$3$}
\put(52,6){\circle*{1.3}}\put(51.2,2){$2$}
\put(20,4){\line(1,1){17}}\put(37,21){\circle*{1.3}}
\put(37,21){\line(1,-1){15}}\put(36.5,22){$7$}
\put(46,12){\circle*{1.3}}\put(44,9){$4$}
\put(52,6){\line(1,1){25}}\put(73,27){\circle*{1.3}}\put(72.5,23.5){$8$}
\put(77,31){$+\infty$}
\end{picture}
\caption{MFS-actions on $65137428$
\label{valhop}}
\end {figure}

 \begin{lemma}\label{lem:act}
 Let $\sigma\in\S_n$.  For each $x\in[n]$, we have 
  $\ai(\sigma)=\ai(\varphi'_x(\sigma))$.
  \end{lemma}
 \begin{proof}
 If $x$ is a peak or a valley of $\sigma$, then $\varphi'_x(\sigma)=\sigma$ and the result is true. If $x$ is a double descent of $\sigma$, then $w_2$ is empty in the $x$-factorization of 
 $\sigma=w_1xw_3w_4$ and there is no admissible inversions of $\sigma$ between $x$ and the word $w_3$. As $\varphi'_x(\sigma)=w_1w_3xw_4$,
  there is no inversions of $\varphi'_x(\sigma)$ between the word $w_3$ and $x$. Let $(\sigma_i,\sigma_j)\notin\{(x,y) : \text{$y$ is a letter in $w_3$}\}$ be a pair of $\sigma$ such that $i<j$. We claim that $(\sigma_i,\sigma_j)$ is an admissible inversion of $\sigma$ if and only if it is an  admissible inversion of $\varphi'_x(\sigma)$, from which the result follows. 
  
  For a word $w$, we write $a\in w$ if $a$ is a letter in $w$. To check the claim, there are $6$ cases to be considered: (1) $\sigma_i\in w_1$ and $\sigma_j\in w_1$; (2) $\sigma_i\in w_1$ and $\sigma_j\in xw_3$; (3) $\sigma_i\in w_1$ and $\sigma_j\in w_4$; (4) $\sigma_i\in w_3$ and $\sigma_j\in w_3$; (5) $\sigma_i\in xw_3$ and $\sigma_j\in w_4$; (6) $\sigma_i\in w_4$ and $\sigma_j\in w_4$. We will only show case (2), other cases are similar.  If $(\sigma_i,\sigma_j)$ is an admissible inversion of $\sigma$, then $\sigma_{i-1}<\sigma_i>\sigma_j$ or $\sigma_j<\sigma_i<\sigma_k$ for some $i<k<j$. Clearly, $(\sigma_i,\sigma_j)$ is an admissible of $\varphi'_x(\sigma)$ if $\sigma_k\ne x$. Otherwise if $\sigma_k=x$, then we denote $x'$ the last letter of $w_1$ and consider the triple $(\sigma_i,x',\sigma_j)$. This indicates that $(\sigma_i,\sigma_j)$ is an admissible inversion of $\varphi'_x(\sigma)$, since $x'>x>\sigma_i$. To show that, if $(\sigma_i,\sigma_j)$ is an admissible inversion of $\varphi'(\sigma)$ then $(\sigma_i,\sigma_j)$ is an admissible inversion of $\sigma$, is similar and we omit. This finishes the proof of our claim in case (2).
 \end{proof}

 For any  permutation $\sigma\in\S_n$,
    let $\Orb(\sigma)=\{g(\sigma): g\in\Z_2^n\}$ be the orbit of $\sigma$ under the MFS-action. 
The MFS-action divides the set $\S_n$ into disjoint orbits. 
 Moreover, for $\sigma\in \S_n$, if $x$ is a double descent of $\sigma$, then 
 $x$ is a double ascent of $\varphi'_x(\sigma)$.
 Hence, there is a unique  
  permutation in each orbit which has no 
 double descent.
 Now, let 
  $\bar\sigma$ be such a unique element  in $\Orb(\sigma)$, then
  $\da(\bar{\sigma})=n-\peak(\bar{\sigma})-\valley(\bar{\sigma})$ and 
   $\des(\bar{\sigma})=\peak(\bar{\sigma})=\valley(\bar{\sigma})-1$.
Thus
$$
\sum_{\pi\in\Orb( \sigma)}q^{\ai(\pi)}t^{\des(\pi)}=
q^{\ai(\bar{\sigma})}t^{\des(\bar{\sigma})}(1+t)^{\da(\bar{\sigma})}
=q^{\ai(\bar{\sigma})}t^{\des(\bar{\sigma})}(1+t)^{n-1-2\des(\bar{\sigma})}.
$$
Therefore, by \eqref {des:ai}, we have 
$$
A_n(t,1,q)=\sum_{\sigma\in \S_n}t^{\des(\sigma)}q^{\ai(\sigma)}=\sum_{k=0}^{\lfloor(n-1)/2\rfloor}\biggl(\:\sum_{\sigma\in \D_{n,k}}p^{\ai(\sigma)}\biggr)t^k(1+t)^{n-1-2k}.
$$
Thus Theorem~\ref{main:result} is a consequence of  
the following result.

 \begin{lemma}\label{lem:inv}
 For each $\sigma\in \D_{n,k}$, we have
 $
 \ai(\sigma)=\inv(\sigma).
 $
 \end{lemma}
  \begin{proof}
  Let $\sigma\in\S_n$ be a permutation without double descent. Let $(\sigma_i,\sigma_j)$ be an inversion of $\sigma$. If $i=1$, then $\sigma_1<\sigma_{2}$ and so $(\sigma_i,\sigma_j)$ is an admissible inversion of $\sigma$. Assume that  $i\geq2$.
If $\sigma_{i-1}<\sigma_i$ then 
$(\sigma_i,\sigma_j)$ is  an admissible inversion of $\sigma$;
if $\sigma_{i-1}>\sigma_{i}$, then we must have $\sigma_{i+1}>\sigma_i$, otherwise $\sigma_i$ will be a double descent of $\sigma$, hence  $(\sigma_i,\sigma_j)$ is also an admissible inversion of $\sigma$. This shows  that $\ai(\sigma)\geq\inv(\sigma)$. As 
 $\ai(\sigma)\leq\inv(\sigma)$ by definition, we are done.
 \end{proof}

\section{Rix-factorization}\label{sec:fac}
We first give a slightly less recursive description  of the $``\rix"$ statistic. Let  $w=w_1\cdots w_k$ be  a word with distinct letters over the alphabet $[n]$.  We say that $w_i$  is a {\em descent top} of $w$ if $w_i>w_{i+1}$. If $k=1$ the word $w$
is called  a {\LG},  if $k\geq2$ the word $w$  is called a {\LG} (resp.~{\FG}) 
if the last (resp. first) letter $w_k$ (resp.~$w_1$) is the 
greatest letter of $w$.

\begin{definition}[Rix--factorization]
Each permutation $\sigma\in\S_n$ can be factorized as
\begin{equation}\label{quasi}
\sigma=\alpha_1\alpha_2\cdots\alpha_{i}\beta,
\end{equation}
where each $\alpha_j$ ($1\leq j\leq i$) is a {\LG} of length $\geq2$ (possibly empty)  and $\beta$ is  a {\LG} or \FG,
by applying the following algorithm: 
\begin{itemize}
\item[(i)] $w\leftarrow\sigma$; $i\leftarrow0$;
\item[(ii)] if $w$ is an increasing word, let $\beta=w$ and we get~\eqref{quasi}; 
otherwise,   $i\leftarrow i+1$, let $x$ be the greatest descent  top  of $w$ and write 
 $w=w'xw''$ for some subwords $w',w''$;
\item[(iii)] if $w'=\emptyset$, let $\beta=w$ and 
we get~\eqref{quasi}; otherwise,  let $\alpha_i=w'x$, 
$w\leftarrow w''$ and go to (ii).
 \end{itemize}
 The factorization  \eqref{quasi} will be called the {\em rix-factorization} of $\sigma$.
 Denote by $\beta_1(\sigma)$ the first letter of $\beta$.
Then a letter in the maximal increasing suffix 
of $\beta$ that is not smaller than $\beta_1(\sigma)$ will be called a {\em rixed point} of $\sigma$. Denote by $\Rix(\sigma)$ the set of all rixed points of $\sigma$.
 \end{definition}
 For example, the algorithm gives 
\begin{align*}
\sigma :&=2\,1\,8\,7\,9\,3\,5\,4\,6\,10=2\,1\,8\,7\,9|3\,5|\beta \quad 
\text{with}\quad \beta=4\,6\,10\;(\text{\LG}),\\
\tau:&=6\,1\,10\,8\,4\,9\,7\,2\,5\,3=6\,1\,10|8\,4\,9|\beta\quad\,\text{with}\quad 
\beta=7\,2\,5\,3 \;(\text{\FG}),\\
\pi:&=1\,10\,4\,7\,6\,2\,5\,3\,8\,9=1\,10|\,4\,7|\beta\qquad\,\,\text{with}\quad 
\beta=6\,2\,5\,3\,8\,9 \;(\text{\LG}).
\end{align*}
We have 
\begin{align*}
\beta_1(\sigma)&=4,\quad \Rix(\sigma)=\{4,6,10\};\\
\beta_1(\tau)&=7, \quad \Rix(\tau)=\emptyset;\\
\beta_1(\pi)&=6,\quad \Rix(\pi)=\{8,9\}.
\end{align*}

\begin{proposition}\label{pro} Let  $\sigma\in\S_n$ with 
 rix-factorization \eqref{quasi}. If  $x_k$ is the last letter  of $\alpha_k$ ($1\leq k\leq i$), then 
\begin{itemize}
\item[(i)] $x_1>x_2>\cdots>x_i>\beta_1(\sigma)$.
\item[(ii)] $\rix(\sigma)=0$ if and only if $\beta$ is an \FG.
\item[(iii)] $\rix(\sigma)=|\Rix(\sigma)|$. 
\end{itemize}
\end{proposition}
\begin{proof} 
(i)  For $k=1,\ldots, i$,  since $x_k$ is the greatest descent top of $\alpha_k\cdots\alpha_{i}\beta$,
it  is obvious that $x_k>x_{k+1}$, where $x_{i+1}=\beta_1(\sigma)$.
(ii) If $\beta$ is an \FG{} (so the length of $\beta$ is at least 2),  when applying  the algorithmic definition to $\sigma$ for  computing 
the statistic $``\rix"$,
 the successive greatest letters are $x_1, x_2, \ldots, x_i$ and $\beta_1(\sigma)$, 
so $\rix(\sigma)=0$ because $\beta_1(\sigma)$ is the first letter of $\beta$. The "only if" part will follow from  
(iii). Indeed, if $\beta$ is  an \LG, when applying  the algorithmic definition to $\sigma$ for  computing 
the statistic $``\rix"$, 
the geatest letters are the letters in 
$$
\Rix(\sigma)\cup \{x_1, x_2, \ldots, x_i,  \beta_1(\sigma)\},
$$ 
read in decreasing order, but only the letters in  
$\Rix(\sigma)$ will contribute to the statistic $``\rix"$.
\end{proof}

 The rix-factorization is reminiscent  of Gessel's hook factorization~\cite{ge}, which decomposes each permutation into pieces of  hooks, uniquely.  Indeed we have the following result.
 \begin{proposition}
 The rix-factorization~\eqref{quasi} is unique provided  condition~(i) in Proposition~\ref{pro} holds and $\beta_1(\sigma)$ is the greatest descent top of $\beta$ (whenever $\beta$ is not an increasing word).
 \end{proposition}
 \begin{proof}
 One can show that each {L-hook} or {F-hook} must be factored out  step by step, from left to right, by applying the algorithm of rix-factorization. 
 \end{proof}

For each permutation $\sigma\in\S_n$ we define its {\em standard cycle form} ($\SCF$)
 as the product of orderd cycles satisfying the following conditions:
\begin{itemize}
\item each cycle $\big(a\,\sigma(a)\,\dots\,\sigma^l(a)\big)$,  with $\sigma^{l+1}(a)=a$, is written with its largest element first $a\geq \sigma^k(a)$ for all $k\geq 1$;
\item the cycles of length $\geq2$ are arranged  in decreasing order of their largest element;
\item the cycles of length $1$ are arranged in increasing order of their elements and 
after  the cycles of length $\geq2$.
\end{itemize}
For example, if
$\sigma=129753468$, then $\SCF(\sigma)=(9\,8\,6\,3)(7\,4)(1)(2)(5)$.

 Let  $\Fix(\sigma)$ be the set of all  fixed points of $\sigma\in\S_n$ and 
 introduce the set  
 \begin{align}\label{newmodel}
 \R^{0}_{n,k}:=\{\sigma\in \S_n : \rix(\sigma)=0,\dd(\sigma)=1,\des(\sigma)=k\}.
 \end{align}
 \begin{proposition}\label{Phi}
  There is  a bijection  
$\Phi: \S_n\rightarrow\S_n$ satisfying
$$
\des(\sigma)=\exc(\Phi(\sigma))\quad\text{and} \quad\Rix(\sigma)=\Fix\big(\Phi(\sigma)\big)
$$
 for each $\sigma\in\S_n$. 
Moreover, the restriction of $\Phi$ on $\R_{n,k}^0$ is  a bijection from  $\R_{n,k}^0$ to  $\DE_{n,k}$.
 \end{proposition}
 \begin{proof}
Starting from $\sigma\in \S_n$ with  rix-factorization $\alpha_1\alpha_2\cdots\alpha_{i}\beta$.  Deleting  all the rixed points of $\sigma$ we  get 
 $$
 \sigma'=\alpha_1\alpha_2\cdots\alpha_{i}\beta',
 $$
 where $\beta'$ is the word (may be empty) obtained from $\beta$ by removing all the rixed points. Clearly, $\beta'$ is a \FG.
 To  any {\LG} or {\FG} $\alpha=a_1a_2\cdots a_l$, we associate the cycle 
 $$
 \tilde{\alpha}=
 \begin{cases}
 (a_l,a_{l-1},\ldots,a_1),\qquad&\text{if $\alpha$ is a \LG;}\\
 (a_1,a_{l},a_{l-1},\ldots,a_2),\qquad&\text{if $\alpha$ is a \FG.}
 \end{cases}
 $$
Now, we define the SCF of $\Phi(\sigma)$ by
 $$
 \Phi(\sigma):=\tilde{\alpha_1}\tilde{\alpha_2}\cdots\tilde{\alpha_{i}}\tilde{\beta'}(\sigma_k)(\sigma_{k+1})\cdots(\sigma_n),
 $$
 where $\Rix(\sigma)=\{\sigma_k,\sigma_{k+1},\ldots,\sigma_n\}$.
 By Proposition~\ref{pro}, the mapping $\Phi$ is well-defined.
For example, if 
 $$
  \sigma=7\,6\,9\,1\,8\,4\,2\,3\,5\,10=7\,6\,9|1\,8|4\,2\,3\,5\,10,
 $$
 then $\sigma'=7\,6\,9|1\,8|4\,2\,3$ and thus 
 $$\Phi(\sigma)=(9,6,7)(8,1)(4,3,2)(5)(10)=8\,4\,2\,3\,5\,7\,9\,1\,6\,10.$$
 Clearly, $\des(\sigma)=4=\exc(\Phi(\sigma))$ and $\Rix(\sigma)=\{5,10\}=\Fix(\Phi(\sigma))$. 
 
 To show that $\Phi$ is a bijection, we define explicitly its inverse.
 Given a permutation $\sigma\in\S_n$, we write $\sigma$ in standard cycle form as 
$$
 \sigma=C_1C_2\cdots C_i O_1O_2\cdots O_j,
$$
 where the $C$'s are cycles of length $\geq2$ and $O_k=(o_k)$ is one point cycle for $k=1,2,\ldots,j$. For each cycle $C=(c_1,c_2,\ldots,c_l)$ of length $l\geq2$, we define the two words ({\LG} or {\FG})
 $$
 \tilde{C}:=c_lc_{l-1}\cdots c_1\quad\text{and}\quad C':=c_1c_lc_{l-1}\cdots c_2.
 $$
 Then $\Phi^{-1}(\sigma)$ (viewed as a word) is defined as
  $$
 \Phi^{-1}(\sigma)=
 \begin{cases}
 \tilde{C_1}\tilde{C_2}\cdots C'_io_1\cdots o_j,\qquad&\text{if $o_1>$ the largest element of $C_i$;}\\
 \tilde{C_1}\tilde{C_2}\cdots \tilde{C_i}o_1\cdots o_j,\qquad&\text{otherwise.}
 \end{cases}
 $$
 For example, $\Phi^{-1}((9,6,7)(8,1)(4,3,2)(5)(10))=7\,6\,9\,1\,8\,4\,2\,3\,5\,10$.
 
  It is straightforward to check the desired properties of $\Phi$.
  \end{proof}

 \begin{remark}
 In view of~\eqref{des:ai}, the two triples $(\exc,\fix,\maj)$ and $(\des,\rix,\aid)$ are equidistributed on $\S_n$, where $\aid(\sigma)=\ai(\sigma)+\des(\sigma)$. 
 Note that $(\Fix,\maj)$ 
 and $(\Rix,\aid)$ are 
 not  equidistributed on $\S_3$.
 \end{remark}

It is easy to construct a bijection between  $\R^0_{n,k}$ and $\widetilde{\D}_{n,k}$.
Define the mapping 
\begin{equation}\label{bij:f}
f: \R^0_{n,k}\rightarrow\widetilde{\D}_{n,k}\quad\text{by}\quad  
 f(\sigma)=\varphi'_x(\sigma),
 \end{equation}
 where $\sigma$ is any permutation in $\R^0_{n,k}$  and $x=\beta_1(\sigma)$.
 
 \begin{corollary}
The mapping $f: \R^0_{n,k} \to \widetilde{\D}_{n,k}$ is a bijection.
 \end{corollary}
 \begin{proof}
 For each $\sigma\in \R^0_{n,k}$, by Proposition~\ref{pro} the letter $\beta_1(\sigma)$ is the only double descent of $\sigma$, which becomes a double ascent of $f(\sigma)$. As $\rix(\sigma)=0$, the last letter of $f(\sigma)$ is $\beta_1(\sigma)$. Thus $f(\sigma)$ is a permutation in $\widetilde{\D}_{n,k}$ and $f$ is well-defined. To show  that $f$ is a bijection we define its inverse $f^{-1}$ as  the word
 $$
 f^{-1}(\sigma):=\varphi'_y(\sigma),
 $$
where $\sigma$ is any permutation in $\widetilde{\D}_{n,k}$ and $y$ is the last letter of $\sigma$.
 All we need to check is that now $f^{-1}(\sigma)$ is a permutation in $\R^0_{n,k}$. Clearly, the letter $y$ becomes the only double descent in $ f^{-1}(\sigma)$ and $\des(f^{-1}(\sigma))=k$. Moreover, we have $y=\beta_1(f^{-1}(\sigma))$ by the algorithm of the rix-factorization, and so $\rix(f^{-1}(\sigma))=0$. Therefore, the word $f^{-1}(\sigma)$ is a permutation in $\R^0_{n,k}$.
\end{proof}
\begin{remark}\label{bijection}
In view of Proposition~\ref{Phi}, the composition $\Phi\circ f^{-1}$ is a bijection between $\widetilde{\D}_{n,k}$ and $\DE_{n,k}$.  This bijection is illustrated in Table~\ref{phi:f}.
\end{remark}
 
\begin{table}[ht]
\[ \begin{tabular}{c |c c ccc}
$\widetilde{\D}_{4,2}$ &\qquad $1324$ &\quad $1423$ &\quad $2314$ &\quad $2413$  &\quad $3412$
\\
\hline
\\
$\R^0_{4,2}$ &\qquad $4132$ &\quad$1432$ &\quad$4213$ &\quad $2431$  &\quad$3421$ 
\\
\hline
\\
$\DE_{4,2}$  &\qquad $4312$ &\quad $4321$&\quad $2413$&\quad $3412$&\quad $2143$\\
\hline
 \end{tabular} \] 
\caption{The bijection $\Phi\circ f^{-1}: \widetilde{\D}_{4,2}
\stackrel{f^{-1}}{\rightarrow}\R_{4,2}\stackrel{\Phi}{\rightarrow}  \DE_{4,2}$. \label{phi:f}}
\end{table}

\section{Proof of Theorem~\ref{main:result2}}\label{sec:main2}
For any $x\in[n]$, we introduce a new action $\varphi''_x:\S_n\rightarrow\S_n$ as
$$
\varphi''_x(\sigma)=
\begin{cases}
\sigma,&\text{if $x\in \{\beta_1(\sigma)\}\cup \Rix(\sigma)$;}\\
\varphi'_x(\sigma),& \text{otherwise.}
\end{cases}
$$
Since $\varphi'_x$'s are involutions and that they commute, so do $\varphi''_x$'s. Thus, for any subset $S\subseteq[n]$ we can define the mapping  $\varphi''_S :\S_n\rightarrow\S_n$ by 
$$
\varphi''_S(\sigma)=\prod_{x\in S}\varphi''_x(\sigma).
$$
This is a new $\Z_2^n$-action on $\S_n$ via the functions $\varphi''_S$, $S\subseteq[n]$, that we called {\em restricted MFS-action}. An important property of the action $\varphi''_x$ is that it preserves the 
rix-factorization type of permutations, namely, 
\begin{lemma}
Let $\sigma\in\S_n$ be  a permutation 
with rix-factorization $\alpha_1\alpha_2\cdots\alpha_i\alpha_{i+1}$ with $\alpha_{i+1}=\beta$. 
For any $x\in[n]$, 
if $x$ is a letter of $\alpha_k$ ($1\leq k\leq i+1$) , then
the rix-factorization of $\varphi''_x(\sigma)$ is 
$\alpha_1\ldots\alpha'_k\ldots \alpha_i\alpha_{i+1}$, 
where $\alpha'_k$ is a rearrangement of $\alpha_k$.
 Moreover,
\begin{equation}\label{var:rix}
\beta_1(\varphi''_x(\sigma))=\beta_1(\sigma)\quad\text{and}\quad \Rix(\varphi''_x(\sigma))= \Rix(\sigma).
\end{equation}
\end{lemma}
\begin{proof}
For $x\in [n]$ we apply $\varphi_x''$ to $\sigma$ and  distinguish two cases.
\begin{itemize}
\item[a)] $x$ is a letter of $\alpha_k$ ($1\leq k\leq i$). Let $x_k$ be the last letter of $\alpha_k$, which is also the greatest letter in $\alpha_k$.
 If $x=x_k$,  then $x$ is a peak of $\sigma$ by Proposition~\ref{pro}, so $\varphi_x''(\sigma)=\sigma$. 
 If $x\neq x_k$, then $x<x_k<x_{k-1}$  by Proposition~\ref{pro},
where  $x_0=+\infty$.
Therefore, 
$\varphi_x''(\sigma)=\alpha_1\ldots \alpha_{k-1}\varphi_x'(\alpha_k)\alpha_{k+1}\ldots \alpha_i\beta$ with  $x_k$ as the last letter of $\varphi_x'(\alpha_k)$.
\item[b)] $x$ is a letter of $\alpha_{i+1}=\beta$. If $x\in \{\beta_1(\sigma)\}\cup \Rix(\sigma)$, then $\varphi''_x(\sigma)=\sigma$; otherwise,
$x<\beta_1(\sigma)$. Therefore, 
$\varphi_x''(\sigma)=\alpha_1\ldots \alpha_i\varphi_x'(\beta)$ with  $\beta_1(\sigma)$ as the first  letter of $\varphi_x'(\beta)$.
\end{itemize}
In both cases, the result  follows from the rix-factorization.
\end{proof}

Let  $\mathcal{R}_n$ be  the set of permutations in $\S_n$ without rixed points.

\begin{lemma}\label{exp:rix}
Let $``\st"$ be a statistic invariant  under the restricted MFS-action. 
Then,
$$
\sum_{\sigma\in\mathcal{R}_n}q^{\st(\sigma)}t^{\des(\sigma)}=\sum_{k=1}^{\lfloor n/2\rfloor}\biggl(\:\sum_{\sigma\in \R^0_{n,k}}q^{\st(\sigma)}\biggr)t^k(1+t)^{n-2k}.
$$
\end{lemma}
\begin{proof}
First we note that the restricted MFS-action is stable on $\mathcal{R}_n$, that is,  $\sigma\in\mathcal{R}_n$ implies $\varphi''_x(\sigma)\in\mathcal{R}_n$ for any $x\in[n]$. This follows from the property of $\varphi''_x$ in~\eqref{var:rix}.

For a permutation $\sigma\in\mathcal{R}_n$ let $\widetilde{\Orb}(\sigma)=\{g(\sigma): g\in\Z_2^n\}$ be the orbit of $\sigma$ in $\mathcal{R}_n$ under the restricted MFS-action. By Proposition~\ref{pro}, the letter $\beta_1(\sigma)$ is a double descent of $\sigma$. If $x\in[n]$ is a double descent of $\sigma$ different from $\beta_1(\sigma)$, then $x$ becomes a double ascent of $\varphi''_x(\sigma)$ and $\des(\sigma)=\des(\varphi''_x(\sigma))+1$. Therefore, there is a unique element $\tilde{\sigma}\in\widetilde{\Orb}(\sigma)$ 
 with  one double descent $\beta_1(\tilde{\sigma})$ such that
$$
\sum_{\pi\in\widetilde{\Orb}(\sigma)}q^{\st(\pi)}t^{\des(\pi)}=q^{\st(\tilde{\sigma})}t^{\des(\tilde{\sigma})}(1+t)^{\da(\tilde{\sigma})}.
$$
Since, for any  $\sigma\in \S_n$,  we have  $\valley(\sigma)=\peak(\sigma)+1$ and
$\dd(\sigma)+\da(\sigma)+2\peak(\sigma)=n-1$,
the result follows then from the fact that
$\dd(\tilde{\sigma})=1$ and $\peak(\tilde{\sigma})=\des(\tilde\sigma)-1$.
\end{proof}

It follows from Eq.~\eqref{des:ai} and Lemmas~\ref{lem:act} and~\ref{exp:rix} that 
$$
A_n(t,0,q)=\sum_{\sigma\in\mathcal{R}_n}t^{\des(\sigma)}q^{\ai(\sigma)}=\sum_{k=1}^{\lfloor n/2\rfloor}\biggl(\:\sum_{\sigma\in \R^0_{n,k}}q^{\ai(\sigma)}\biggr)t^k(1+t)^{n-2k}.
$$
 Applying the bijection $f: \R^0_{n,k}\rightarrow\widetilde{\D}_{n,k}$ defined in~\eqref{bij:f}, we get
 $$
 A_n(t,0,q)=\sum_{k=1}^{\lfloor n/2\rfloor}\biggl(\:\sum_{\sigma\in \widetilde{\D}_{n,k}}q^{\ai(\sigma)}\biggr)t^k(1+t)^{n-2k},
 $$
 which is equivalent to expansion~\eqref{der2} in view of Lemma~\ref{lem:inv}.
 
 \section{Generating functions of $\gamma$-coefficients}\label{sec:gamma-coeff}

In this section, we derive the recurrences and generating functions of 
the two $\gamma$-coefficients in~\eqref{gam:euler} and~\eqref{gam:der} from 
their combinatorial  interpretations~\eqref{per} and \eqref{der2}, namely,
\begin{align}\label{def:gamma}
\gamma_{n,k}(q)=\sum_{\sigma\in \D_{n,k}}q^{\inv(\sigma)} \quad\text{and}\quad\widetilde{\gamma}_{n,k}(q)=\sum_{\sigma\in \widetilde{\D}_{n,k}}q^{\inv(\sigma)}.
\end{align}
Let  $\D_n=\{\sigma\in\S_n: \dd(\sigma)=0\}$ and 
 $\widetilde{\D}_n=\{\sigma\in\D_n: \sigma_{n-1}<\sigma_n\}$. 
Then
\begin{align}\label{int:gamma}
\Gamma_n(y,q):&=\sum_{k=0}^{\lfloor(n-1)/2\rfloor}\gamma_{n,k}(q)y^k=\sum_{\sigma\in \D_n}y^{\des(\sigma)}q^{\inv(\sigma)};\\
\widetilde{\Gamma}_n(y,q):&=\sum_{k=1}^{\lfloor n/2\rfloor}\widetilde{\gamma}_{n,k}(q)y^k=\sum_{\sigma\in \widetilde{\D}_n}y^{\des(\sigma)+1}q^{\inv(\sigma)}.
\end{align}
The  first values of $\Gamma_n(y,q)$ and $\widetilde{\Gamma}_n(y,q)$  for $1\leq n\leq5$ are:
\begin{align*}
\Gamma_1(y,q)&=\Gamma_2(y,q)=1,\quad \Gamma_3(y,q)=1+y(q+q^2),\\
\Gamma_4(y,q)&=1+y(q+q^2)(2+q+q^2),\\
\Gamma_5(y,q)&=1+y(3q+5q^2+5q^3+5q^4+2q^5+2q^6)\\
&+y^2(q+q^3)(1+q+q^2+q^3)(q+q^2);
\end{align*}
and
\begin{align*}
\widetilde{\Gamma}_1(y,q)&=0,\quad \widetilde{\Gamma}_2(y,q)=\widetilde{\Gamma}_3(y,q)=y,\\
\widetilde{\Gamma}_4(y,q)&=y+y^2(q+2q^2+q^3+q^4),\\
\widetilde{\Gamma}_5(y,q)&=y+y^2(2q+4q^2+4q^3+4q^4+2q^5+2q^6).
\end{align*}

For $n\geq 1$ let
$(q;q)_n :=\prod_{i=1}^{n}(1-q^i)$  be the shifted $q$-factorial and $(q;q)_0=1$. The 
{\em $q$-exponential function} and {\em $q$-binomial coefficients} are then defined by
$$
e(z;q):=\sum_{n\geq 0}\frac{z^n}{(q;q)_n}\quad \text{and}\quad
 {n\brack  k}_{q}:=\frac{(q;q)_n}{(q;q)_k(q;q)_{n-k}}.
$$

The following 
  generating function formula is due to Shareshian and Wachs~\cite{sw2,sw}:
\begin{equation}\label{fixversion}
1+\sum_{n\geq1}A_n(t,r,q)\frac{z^n}{(q;q)_n}=\frac{(1-t)e(rz;q)}{e(tz;q)-te(z;q)}.
\end{equation}

\begin{proposition}\label{gen:gamma}
The exponential generating functions  for $\Gamma_n(y,q)$ and 
$\widetilde{\Gamma}_n(y,q)$ are
\begin{align}
\sum_{n\geq0}\Gamma_n(y,q)\frac{z^n}{(q;q)_n}&=
\frac{e(z/(1+t);q)-te(tz/(1+t);q)}{e(tz/(1+t);q)-te(z/(1+t);q)},\label{gf1}\\
\sum_{n\geq0}\widetilde{\Gamma}_n(y,q)\frac{z^n}{(q;q)_n}&=
\frac{1-t}{e(tz/(1+t);q)-te(z/(1+t);q)},\label{gf2}
\end{align}
where $\Gamma_0(y,q)=\widetilde{\Gamma}_0(y,q)=1$  
and $y=t/(1+t)^2$.
\end{proposition}
\begin{proof} 
 Clearly  Theorems~\ref{main:result} and \ref{main:result2} are   equivalent to, for $n\geq1$,
\begin{align}\label{rela:gam}
A_n(t,1,q)
&=(1+t)^{n-1}\Gamma_n(y,q),\\
  A_n(t,0,q)&=(1+t)^{n}\tilde{\Gamma}_n(y,q),
  \label{rela:gamtilde}
\end{align}
where $y=t/(1+t)^2$. Substituting~\eqref{rela:gam} and \eqref{rela:gamtilde}  into~\eqref{fixversion} 
yields \eqref{gf1} and \eqref{gf2}.
\end{proof}

Recall  the well-known combinatorial interpretation of the $q$-binomial coefficients
~\cite[Prop.~1.3.17]{st0}
\begin{equation}\label{eq:qmul}
{n\brack  k}_{q}=\sum_{({\mathcal A}, {\mathcal B})}q^{\inv({\mathcal A}, {\mathcal B})},
\end{equation}
where the sum is over all ordered set partitions $({\mathcal A}, {\mathcal B})$ of $[n]$ such that $|{\mathcal A}|=k$ and 
$$
\inv({\mathcal A}, {\mathcal B}):=\{(i,j)\in {\mathcal A}\times {\mathcal B}: i>j\}.
$$

\begin{proposition} \label{rec:gamma}
The polynomials $\Gamma_n(y,q)$ and 
$\widetilde{\Gamma}_n(y,q)$ satisfy the following recurrences:
\begin{align}\label{eq:gamm}
\Gamma_{n+1}(y,q)&=\Gamma_{n}(y,q)+y\sum_{i=1}^{n-1}q^i{n\brack i}_q\Gamma_i(y,q)\Gamma_{n-i}(y,q);\\
\widetilde{\Gamma}_{n+1}(y,q)&=\Gamma_{n}(y,q)+y\sum_{i=2}^{n-1}q^i{n\brack i}_q\widetilde{\Gamma}_i(y,q)\Gamma_{n-i}(y,q);\label{eq:gamm2}
\end{align}
for $n\geq1$ and the initial conditions 
$\Gamma_0(t,q)=\Gamma_1(t,q)=\widetilde{\Gamma}_0(t,q)=1$ and 
$\widetilde{\Gamma}_1(t,q)=0$. 
\end{proposition}
\begin{proof}
For $2\leq j\leq n$ let $\D^{(j)}_n:=\{\sigma\in \D_n: \sigma_j=n\}$
and define
$$
\Gamma^{(j)}_n(y,q):=\sum_{\sigma\in \D^{(j)}_n}y^{\des(\sigma)}q^{\inv(\sigma)}.
$$ 
Clearly, for $n\geq1$
\begin{equation}\label{sumB}
\Gamma_{n+1}(y,q)=\sum_{j=2}^{n+1}\Gamma^{(j)}_{n+1}(y,q)=\Gamma_n(y,q)+\sum_{j=2}^n\Gamma^{(j)}_{n}(y,q).
\end{equation}
For any finite ordered  set $X$ let ${X \choose m}$ denote the set of the $m$-element subsets of 
$X$ and $\D_X$ the set of permutations of $X$ without double descents.
Also let $\T(n,j)$ be the set of all triples $(S,\sigma_L,\sigma_R)$ such that $S\in{[n] \choose j}$ and $\sigma_L\in\D_S$, $\sigma_R\in \D_{[n]\setminus S}$.   
For $2\leq j\leq n-1$ and $\sigma\in \D^{(j)}_n$,  define the mapping $\sigma\mapsto (S,\sigma_L,\sigma_R)$ by
\begin{itemize}
\item $S=\{\sigma_i : 1\leq i\leq j-1\}$;
\item $\sigma_L=\sigma_1\sigma_2\cdots\sigma_{j-1}$ and $\sigma_R=\sigma_{j+1}\sigma_{j+2}\cdots\sigma_n$.
\end{itemize}
It is not hard to see that this mapping is a bijection between $\D_n^{(j)}$ and $\T(n-1,j-1)$ and satisfies 
$$
\des(\sigma)=\des(\sigma_L)+\des(\sigma_R)+1 
$$
and
$$
\inv(\sigma)=\inv(\sigma_L)+\inv(\sigma_R)+\inv(S,[n-1]\setminus S)+n-j.
$$
Thus, for $2\leq j\leq n$ we have
\begin{align}\label{resB}
\Gamma^{(j)}_{n+1}(y,q)&=\sum_{\sigma\in \D^{(j)}_{n+1}}y^{\des(\sigma)}q^{\inv(\sigma)}\nonumber\\
&=yq^{n+1-j}\sum_{(S,\sigma_L,\sigma_R)\in\T(n,j-1)}q^{\inv(S,[n]\setminus S)}q^{\inv(\sigma_L)}y^{\des(\sigma_L)}q^{\inv(\sigma_R)}y^{\des(\sigma_R)}\nonumber\\
&=yq^{n+1-j}\sum_{S\in{[n]\choose j-1}}q^{\inv(S,[n]\setminus S)}\sum_{\sigma_L\in \D_S}q^{\inv(\sigma_L)}y^{\des(\sigma_L)}\sum_{\sigma_R\in \D_{[n]\setminus S}}q^{\inv(\sigma_R)}y^{\des(\sigma_R)}\nonumber\\
&=yq^{n+1-j}{n\brack j-1}_q\Gamma_{j-1}(y,q)\Gamma_{n+1-j}(y,q),
\end{align}
where we apply \eqref{eq:qmul} to the last equality. Substituting~\eqref{resB} into~\eqref{sumB} we obtain~\eqref{eq:gamm}.  Similarly we can prove \eqref{eq:gamm2}, the details are left to the interested reader.
\end{proof}
\begin{remark} We can also prove Proposition~\ref{rec:gamma} from
the following known recurrence of $A_n(t,r,q)$ ($n\geq1$), which was proved in \cite{lin} by applying the $q$-differential operator to both sides of~\eqref{fixversion}, 
\begin{equation}\label{recurrence2}
A_{n+1}(t,r,q)=rA_n(t,r,q)+t\sum_{j=0}^{n-1}{n\brack j}_qq^jA_j(t,r,q)A_{n-j}(t,1,q)
\end{equation}
with  initial condition $A_0(t,r,q)=1,A_1(t,r,q)=r$. Substituting \eqref{rela:gam} and 
\eqref{rela:gamtilde} into \eqref{recurrence2} we also obtain 
the above recurrences for $\gamma$-coefficients.
\end{remark}
\section{Concluding remarks}\label{sec:conclusion}
Using the combinatorics developed in Section~3,
we can also give a combinatorial proof of Theorem~\ref{theoremSZ}, of which a similar proof for the special $\beta=1$ case was given in \cite{as,sun}.

\begin{proof}[Proof of Theorem~\ref{theoremSZ}]
Define the statistic $``\lyc"$ by $\lyc(\sigma):=\cyc(\Phi(\sigma))$ for $\sigma\in\S_n$. Since the function $\varphi''_x:\S_n\rightarrow\S_n$ preserves the rix-factorization type of permutations for any $x\in[n]$, the statistic $``\lyc"$ is  invariant under the restricted MFS-action. Thus by Lemma~\ref{exp:rix}, we have 
$$
\sum_{\sigma\in\R_n}\beta^{\lyc(\sigma)}t^{\des(\sigma)}=\sum_{j=1}^{\lfloor n/2\rfloor}\biggl(\:\sum_{\sigma\in \R^0_{n,j}}\beta^{\lyc(\sigma)}\biggr)t^j(1+t)^{n-2j}.
$$
By Proposition~\ref{Phi},  applying  the bijection $\Phi:\S_n\rightarrow\S_n$ to both sides of the above expansion  we get \eqref{cycle}.
\end{proof}

We can  further extend \eqref{cycle} and \eqref{der2} to 
permutations with  a given  number of fixed points. 
For $0\leq j\leq n$ define  the set
$$\S_{n,j}:=\{\sigma\in\S_n: \fix(\sigma)=j\}.$$
Then,  it is easy to see that 
\begin{align}
\sum_{\sigma\in\S_{n,j}}\beta^{\cyc(\sigma)}t^{\exc(\sigma)}={n\choose j}\beta^j\sum_{\sigma\in\S_{n-j,0}}\beta^{\cyc(\sigma)}t^{\exc(\sigma)}.
\end{align}
Also, it is known \cite[Eq.(4.3)]{sw} and  easy to 
 deduce from \eqref{fixversion} that 
\begin{align}\label{fix-maj}
\sum_{\sigma\in\S_{n,j}}t^{\exc(\sigma)}q^{\maj(\sigma)-\exc(\sigma)}=
{n\brack j}_qA_{n-j}(t,0,q).
\end{align}

We derive from Theorem~\ref{theoremSZ} and 
Theorem~\ref{main:result2} the following result.
\begin{proposition}
For $1\leq j\leq n$ we have 
\begin{equation}\label{cycle-bis}
\sum_{\sigma\in\S_{n,j}}\beta^{\cyc(\sigma)}t^{\exc(\sigma)}=\sum_{k=1}^{\lfloor n/2\rfloor}{n\choose j}\biggl(\:\sum_{\sigma\in \DE_{n-j,k}}\beta^{\cyc(\sigma)+j}\biggr)\,t^k(1+t)^{n-j-2k},
\end{equation}
and 
\begin{equation}\label{exp:fixed}
\sum_{\sigma\in\S_{n,j}}t^{\exc(\sigma)}q^{\maj(\sigma)-\exc(\sigma)}=\sum_{k=1}^{\lfloor (n-j)/2\rfloor}{n\brack j}_q\
\biggl(\:\sum_{\sigma\in \widetilde{\D}_{n-j,k}}q^{\inv(\sigma)}\biggr)
\,t^k(1+t)^{n-j-2k}.
\end{equation}
\end{proposition}
Shareshian and Wachs~\cite[Remark~5.5]{sw} also proved 
the $\gamma$-positivity of the left-hand side of  \eqref{exp:fixed}
and  further generalized 
\eqref{gam:der}   to 
\begin{align}\label{SW3}
\sum_{\sigma\in \S_{n,0}}t^{\exc(\sigma)} p^{\des(\sigma)}q^{\maj(\sigma)-\exc(\sigma)}= \sum_{k=0}^{{\lfloor (n-1)/2\rfloor}} \widetilde{\gamma}_{n,k}(p,q) t^k(1+t)^{n-1-2k},
\end{align}
where $\widetilde{\gamma}_{n,k}(p,q)$ are polynomials in $\N[p,q]$. Obviously we have 
  $\widetilde{\gamma}_{n,k}(1,q)=\widetilde{\gamma}_{n,k}(q)$. 
  It would be interesting  to find a combinatorial interpretation 
  for the polynomials $\widetilde{\gamma}_{n,k}(p,q)$ in the light of
  \eqref{der2}.

\subsection*{Acknowledgement} 
The first author was supported by the China Scholarship Council  
to study abroad  from 2010 to 2014 and would like to thank 
Institut Camille Jordan and 
Universit\'e  Claude Bernard Lyon 1, where most of this work was done,  for hosting him.

\end{document}